\documentclass[11pt]{amsart}
\usepackage{amsmath,amsthm,amsfonts,amssymb,latexsym,epsfig}

\newtheorem{theorem}{Theorem}[section]

\newtheorem{lemma}{Lemma}[section]

\newtheorem{cor}{Corollary}[section]

\numberwithin{equation}{section}

\theoremstyle{definition}

\theoremstyle{remark}

\begin{document}
\title{On Some Double-Series Inequalities}
\author{Peng Gao}
\address{Division of Mathematical Sciences, School of Physical and Mathematical Sciences,
Nanyang Technological University, 637371 Singapore}
\email{penggao@ntu.edu.sg}


\begin{abstract}
  We study certain double-–series
  inequalities, which are motivated by weighted Hardy
  inequalities.
\end{abstract}

\maketitle
\section{Introduction}
\label{sec 1} \setcounter{equation}{0}

  Let $p>0$ and $l^p$ be the space of all complex sequences ${\bf a}=(a_n)_{n \geq 1}$ satisfying:
\begin{equation*}
  \|{\bf a}\|_p=\Big(\sum^{\infty}_{i=1}|a_i|^p \Big )^{1/p}<\infty.
\end{equation*}
   A matrix $A=(a_{n,k})$ is said to be a weighted
    mean matrix if its entries satisfy:
\begin{equation*}
    a_{n,k}=\lambda_k/\Lambda_n,  ~~ 1 \leq k \leq
    n; \hspace{0.1in} \Lambda_n=\sum^n_{i=1}\lambda_i, \lambda_i \geq 0, \lambda_1>0.
\end{equation*}
   For fixed $p>1$, the $l^{p}$ operator norm of $A$ is defined as the $p$-th root of
   the best possible
   constant $U_p$ satisfying:
\begin{align}
\label{1.1}
  \sum^{\infty}_{n=1}\Big{|}\frac {1}{\Lambda_n} \sum^n_{k=1}\lambda_ka_k\Big{|}^p \leq
  U_p\sum^\infty_{n=1}|a_n|^p,
\end{align}
   where the estimate is to hold for all complex
   sequences ${\bf a}$. When $\lambda_k=1$ for all $k$ and $U_p=(p/(p-1))^p$, inequality \eqref{1.1} becomes the celebrated
   Hardy's inequality (\cite[Theorem 326]{HLP}).

   By the duality principle \cite[Lemma 2]{M} for the norms of linear operators,
   inequality \eqref{1.1} is equivalent to the following
\begin{align}
\label{1.2}
  \sum^{\infty}_{n=1}\Big{|}\lambda_n \sum^{\infty}_{k=n}\frac {a_k}{\Lambda_k}\Big{|}^p \leq
  U^{p/q}_q\sum^\infty_{n=1}|a_n|^p,
\end{align}
   where $q=p/(p-1)$.

   From now on we restrict our attention to all non-negative sequences $(a_n)$. Similar to inequality \eqref{1.2}, one can also study the
   following inequality (or its reverse) for various $p$:
\begin{align}
\label{1.3}
    \sum^{\infty}_{n=1}\Big (\frac
   {1}{\Lambda_n}\sum^{\infty}_{k=n}\lambda_ka_k\Big )^p \leq
U_p\sum^{\infty}_{n=1}a^p_n.
\end{align}

   When $0<p<1$ and $\lambda_k=1$ for all $k$, the reversed inequality
   \eqref{1.3} becomes the one studied in Theorem
   345 of \cite{HLP}. The best possible constant $U_p$ in this case is
   not yet known for all $0<p<1$. For studies in this direction,
   we refer the reader to the references \cite[Theorem 61]{L&S}
   and \cite{G9}.

   For fixed $p$, it is interesting to compare the right-hand side expressions
   in \eqref{1.2} and \eqref{1.3}. When $\lambda_k=1$ for all $k$, one has the following result of Bennett and Grosse-Erdmann \cite[Corollary 3]{BGE1}:
\begin{align*}
  \sum^{\infty}_{n=1}\Big (\frac
   {1}{n^{\beta}}\sum^{\infty}_{k=n}a_k\Big )^p \leq
   \frac {1}{1-\beta p}\sum^{\infty}_{n=1}\Big{(}\sum^{\infty}_{k=n}\frac
   {a_k}{k^{\beta}}\Big{)}^p.
\end{align*}
   Here $0<p \leq 1, 0\leq \beta < 1/p$ and the constant is
   best possible.

   More generally, for given matrices $A, B$, one can consider
   inequalities of the type
\begin{align}
\label{1.4}
   \|B{\bf x}\|_p \leq K(p) \|A{\bf x}\|_p,
\end{align}
   where $K(p)$ is a constant, and the estimate is to hold for all non-negative sequences ${\bf x}=(x_n)$. When neither $A$ nor $B$ is a diagonal matrix,
we refer to inequality \eqref{1.4} as double–-series inequality.
The double--series inequalities are studied in \cite{BGE1} and
\cite{Be2}.

  In this paper, we focus on the study of double--series inequalities given in the following form:
\begin{equation}
\label{1}
   \left ( \sum^{\infty}_{n=1} \left ( \sum^{\infty}_{k=n}b_nc_kx_k\right )^p \right )^{1/p} \leq
   K(p,q)\left ( \sum^{\infty}_{n=1} \left ( \sum^{\infty}_{k=1}a_{n,k}x_k\right )^q \right
   )^{1/q}.
\end{equation}
    Here $(a_{n,k}), (b_n)$ and $(c_n)$ are given non-negative
    sequences, $0<p, q < \infty$ are fixed parameters. The estimate is to hold for all non-negative sequences ${\bf x}$. The
    constant $K(p, q)$ is independent of ${\bf x}$.

   We denote $e^{(1)}=(1, 0, 0, \ldots), e^{(2)}=(0,1, 0, \ldots)$ and so on. In Section \ref{sec 3}, we prove the following
\begin{theorem}
\label{thm2} Suppose that $a_{n,k}/c_k$ increases with $k$ for any
fixed $n \geq 1$, then for $p \geq 1, 0<q \leq p$, inequality
\eqref{1} holds for non-negative sequences $x=(x_k)$ if and only
if it holds for $x=e^{(n)} (n=1, 2, \ldots)$. The theorem
continues to hold when $0<p \leq 1, q \geq p$ provided that
inequality \eqref{1} is reversed.
\end{theorem}

    When $p=q$, a special case of Theorem \ref{thm2} is proved in \cite[Lemma 8]{BGE1} while the general case is proved in
    \cite[Lemma 4]{Be2}.

    Motivated by various choices for the $\lambda_k$'s in \eqref{1.3}, we apply Theorem \ref{thm2} in Section \ref{sec 4} to
    determine the best possible constant $K(p)$ with $0<p<1$ in \eqref{1.4} for
    various $A$ and $B$.

\section{Lemmas}
\label{sec 2} \setcounter{equation}{0}
   In this section we list a few lemmas that are needed in the proofs of our results in Section \ref{sec 4}.
   We first note the following lemma of Levin and Ste\v ckin
  \cite[Lemma 1-2, p.18]{L&S}:
\begin{lemma}
\label{lem0}
    For an integer $n \geq 1$,
\begin{eqnarray}
\label{4}
    \sum^n_{i=1}i^r &\geq & \frac {1}{r+1}n(n+1)^r, \hspace{0.1in} 0 \leq r \leq 1, \\
\label{201}
   \sum^n_{i=1}i^r  &\geq & \frac {r}{r+1}\frac
   {n^r(n+1)^r}{(n+1)^r-n^r}, \hspace{0.1in} r \geq 1.
\end{eqnarray}
    Inequality \eqref{4} reverses when $r \geq 1$ or $-1<r \leq 0$. Inequality \eqref{201} reverses when $-1 <r \leq 1$.
\end{lemma}

\begin{lemma}
\label{lem1}
    For $s>r>-1$, $n \geq 1$,
\begin{equation}
\label{4.1}
    \frac {\sum^{n}_{i=1}(i/n)^{r}}{\sum^{n}_{i=1}(i/n)^{s}} <
    \frac {1+s}{1+r}.
\end{equation}
   The constant is best possible.
\end{lemma}
\begin{proof}
     Upon letting $n \rightarrow \infty$, one sees easily that the constant is best possible.  To prove inequality \eqref{4.1},
     we rewrite it as
\begin{equation*}
   f(n):= (1+s)\sum^{n}_{i=1}i^s-(1+r)n^{s-r}\sum^{n}_{i=1}i^r >0.
\end{equation*}
    Note that
\begin{equation*}
    f(n+1)-f(n)=(s-r)(n+1)^s-(1+r)\left((n+1)^{s-r}-n^{s-r} \right )\sum^{n}_{i=1}i^r.
\end{equation*}
    We want to show the above expression is non-negative, which is
    amount to showing
\begin{equation*}
    \sum^{n}_{i=1}i^r \leq \frac {(s-r)(n+1)^s}{(1+r)\left((n+1)^{s-r}-n^{s-r} \right
    )}.
\end{equation*}
   For fixed $r$, it's easy to see that the right-hand side expression above is an
   increasing function of $s>r$ so that it suffices to show
\begin{equation}
\label{3.1}
    \sum^{n}_{i=1}i^r \leq \lim_{s \rightarrow r^+}\frac {(s-r)(n+1)^s}{(1+r)\left((n+1)^{s-r}-n^{s-r} \right
    )}=\frac {(n+1)^r}{(1+r)\ln (1+1/n)}.
\end{equation}
   As it's easy to see inequality \eqref{3.1} follows from
    various cases of inequalities \eqref{4} or \eqref{201}, it follows that $f(n) \geq f(1)=s-r >0$ for all $n \geq
    1$ and this completes the proof.
\end{proof}

\begin{lemma}
\label{lem2}
    For $n \geq 2$,
\begin{align*}
 \frac {\sum^{n-1}_{i=1}(i/n)^{r}}{\sum^{n-1}_{i=1}(i/n)^{s}}  \left\{\begin{array}{ll}
\leq 2^{s-r} & {\mbox{ if }}  s>r\geq 1,   \\
 < \frac {1+s}{1+r} & {\mbox{ if }}  0>s>r>-1.
\end{array}\right.
\end{align*}
  The constants are best possible.
\end{lemma}
\begin{proof}
    We first consider the case $s>r\geq 1$.
   Note that the case $n=2$ implies the constant here is best possible. To prove the
   corresponding inequality, we rewrite it as:
\begin{equation*}
   g(n):= 2^{s-r}\sum^{n-1}_{i=1}i^s-n^{s-r}\sum^{n-1}_{i=1}i^r >0.
\end{equation*}
    Note that
\begin{equation*}
    g(n+1)-g(n)=2^{s-r}n^s-(n+1)^{s-r}n^r-\left((n+1)^{s-r}-n^{s-r} \right )\sum^{n-1}_{i=1}i^r.
\end{equation*}
    We want to show the above expression is non-negative, which is
    amount to showing
\begin{equation}
\label{4.5}
    \sum^{n-1}_{i=1}i^r \leq \frac {2^{s-r}n^s-(n+1)^{s-r}n^r}{(n+1)^{s-r}-n^{s-r}}
    =n^rh\left (s-r;\frac {n}{n+1} \right ),
\end{equation}
   where
\begin{align*}
   h(u;v)=\frac {(2v)^u-1}{1-v^u}.
\end{align*}
    Note that for $u>0, 1/2 \leq v<1$, we have
\begin{align*}
    \frac {\partial h}{\partial u}=\frac {v^u}{(1-v^u)^2}p(u; \ln v).
\end{align*}
   where
\begin{align*}
   p(u; w)=\left (2^u-(2e^w)^u \right )\ln 2+w \left (2^{u}-1 \right ).
\end{align*}
    One sees easily that $p(u;w)$ is a concave function of $w$ for
    fixed $u$ and it follows that $p(u;\ln v) \geq \min (p(u;-\ln
    2), p(u;0))=0$ for $1/2 \leq v<1$.
   We then deduce that in order to establish inequality \eqref{4.5} for $s > r \geq 1$, it
    suffices to show that
\begin{align*}
    \sum^{n-1}_{i=1}i^r \leq \lim_{s \rightarrow r^+}n^rh\left (s-r;\frac {n}{n+1} \right )
    =n^r\left (-1+\frac {\ln 2}{\ln (1+\frac {1}{n})} \right
    ).
\end{align*}
   As the above inequality is an easy consequence of the case $r
   \geq 1$ of inequality \eqref{4}, we see that we have
   $g(n+1)-g(n) \geq 0$ for all $n \geq 2$ and $g(2)=0$, it
   follows that $g(n) \geq 0$ for all $n \geq 2$ and this
   completes the proof for the case $s>r\geq 1$.

   Next, we consider the case $0>s>r>-1$. Upon letting $n \rightarrow
   \infty$, one sees that the constant here is best possible. We prove the
   corresponding inequality by induction. When $n=2$, the inequality follows
    easily from the fact that the function $r \mapsto (1+r)2^{-r}$
    is an increasing function of $-1<r<0$.

    Suppose now the corresponding inequality holds for some $n$ with $n
    \geq 2$, then we have
\begin{equation*}
   \frac {\sum^{n}_{i=1}(i/(n+1))^{r}}{\sum^{n}_{i=1}(i/(n+1))^{s}}
   < (n+1)^{s-r}
    \frac {(\frac
    {1+s}{1+r})n^{r-s}\sum^{n-1}_{i=1}i^{s}+n^r}{\sum^{n}_{i=1}i^{s}}.
\end{equation*}
    It suffices to show that the right-hand side expression above
    is $< \frac {1+s}{1+r}$, which is equivalent to the following
\begin{equation}
\label{4.7'}
    \sum^{n-1}_{i=1}i^s < \frac {n^s}{1+s}\left (-1-s+q\left (s-r;\frac {n}{n+1} \right ) \right
    ),
\end{equation}
   where
\begin{align*}
   q(u;v)=\frac {u}{1-v^u}.
\end{align*}
    It's easy to see that for fixed $0<v<1$, $q(u;v)$ is an
    increasing function of $u>0$.
   It follows that in order to establish inequality \eqref{4.7'} for $0>s>r>-1$, it
    suffices to show that
\begin{align*}
    \sum^{n-1}_{i=1}i^s < \lim_{r \rightarrow s^-}\frac {n^s}{1+s}\left (-1-s+h\left (s-r;\frac {n}{n+1} \right ) \right
    )=\frac {n^s}{1+s}\left (-1-s+\frac {1}{\ln (1+\frac {1}{n})} \right
    ).
\end{align*}
   We now note the reversed inequality \eqref{4} valid for
   $-1<r \leq 0$ implies that
\begin{align*}
    \sum^{n-1}_{i=1}i^s \leq \frac {(n-1)n^s}{1+s}.
\end{align*}
   Thus, it remains to show that
\begin{align*}
    \frac {(n-1)n^s}{1+s} < \frac {n^s}{1+s}\left (-1-s+\frac {1}{\ln (1+\frac {1}{n})} \right
    ).
\end{align*}
   The above inequality is easily seen to be valid by noting that
   $-1<s<0$ and this completes the proof for the case $0>s>r>-1$.
\end{proof}

\begin{lemma}[{\cite[Lemma 3.1]{G7}}]
\label{lem3}
   Let $\{B_n \}^{\infty}_{n=1}$ and $\{C_n \}^{\infty}_{n=1}$ be strictly increasing positive sequences with
   $B_1/B_2 \leq C_1 / C_2$. If for any integer $n \geq 1$,
\begin{equation*}
  \frac {B_{n+1}-B_n}{B_{n+2}-B_{n+1}} \leq  \frac
  {C_{n+1}-C_n}{C_{n+2}-C_{n+1}}.
\end{equation*}
  Then $B_{n}/B_{n+1} \leq C_{n} / C_{n+1}$ for any integer $n \geq 1$.
\end{lemma}

\begin{lemma}
\label{lem4}
    For $1\leq s<r<1/p$,
\begin{equation}
\label{4.11}
    \frac {\sum^{n}_{k=1}(r\sum^{k}_{i=1}i^{r-1})^{-p}}{\sum^{n}_{k=1}(s\sum^{k}_{i=1}i^{s-1})^{-p}}
 < \frac {1-sp}{1-rp}n^{(s-r)p}.
\end{equation}
  The constant $(1-sp)/(1-rp)$ is best possible.
\end{lemma}
\begin{proof}
  We note first that as we have
\begin{align*}
   k^r \leq r\sum^{k}_{i=1}i^{r-1} \leq (k+1)^r,
\end{align*}
   it's easy to see that the constant $(1-sp)/(1-rp)$ in
   \eqref{4.11} is best possible.

   We now prove inequality \eqref{4.11} by induction. Note that when $n=1$, this follows
    easily from the fact that the function $r \mapsto r^p/(1-rp)$
    is an increasing function of $0<r<1/p$.

    Suppose now inequality \eqref{4.11} holds for some $n$ with $n
    \geq 1$, then we have
\begin{equation*}
   \frac {\sum^{n+1}_{k=1}(r\sum^{k}_{i=1}i^{r-1})^{-p}}{\sum^{n+1}_{k=1}(s\sum^{k}_{i=1}i^{s-1})^{-p}}
   <
    \frac {(\frac
    {1-sp}{1-rp})n^{(s-r)p}\sum^{n}_{k=1}(s\sum^{k}_{i=1}i^{s-1})^{-p}+(r\sum^{n+1}_{i=1}i^{r-1})^{-p}}{\sum^{n+1}_{k=1}(s\sum^{k}_{i=1}i^{s-1})^{-p}}.
\end{equation*}
    It suffices to show that the right-hand side expression above
    is $< \frac {1-sp}{1-rp}(n+1)^{(s-r)p}$, which, after simplification, is equivalent to the following
\begin{align*}
   &(1-sp)\left ( \left (1+\frac {1}{n} \right )^{(r-s)p}-1 \right )
   \sum^{n}_{k=1}\left(s\sum^{k}_{i=1}i^{s-1} \right )^{-p}  \\
   < &(1-sp) \left(s\sum^{n+1}_{i=1}i^{s-1} \right)^{-p}-(1-rp)(n+1)^{-sp}\left (\frac
{r}{n+1}\sum^{n+1}_{i=1}\left ( \frac {i}{n+1} \right )^{r-1}
\right )^{-p}.
\end{align*}
    We note that inequality \eqref{4.1} implies that for fixed $n
    \geq 1$, the function
\begin{align*}
    r \mapsto (1+r)\sum^{n}_{i=1}\left ( \frac {i}{n} \right )^r
\end{align*}
   strictly increases with $r>-1$. It follows that we have
\begin{align*}
  & (1-sp) \left(s\sum^{n+1}_{i=1}i^{s-1} \right)^{-p}-(1-rp)(n+1)^{-sp}\left (\frac
{r}{n+1}\sum^{n+1}_{i=1}\left ( \frac {i}{n+1} \right )^{r-1}
\right )^{-p} \\
> & (1-sp) \left(s\sum^{n+1}_{i=1}i^{s-1}
\right)^{-p}-(1-rp)(n+1)^{-sp}\left (\frac
{s}{n+1}\sum^{n+1}_{i=1}\left ( \frac {i}{n+1} \right )^{s-1}
\right )^{-p} \\
=& (r-s)p \left(s\sum^{n+1}_{i=1}i^{s-1} \right)^{-p}.
\end{align*}
   Thus, it remains to show that
\begin{align}
\label{4.12}
  \left(s\sum^{n+1}_{i=1}i^{s-1} \right)^{-p} \geq (1-sp)\frac {\left (1+\frac {1}{n} \right )^{(r-s)p}-1}{(r-s)p}
   \sum^{n}_{k=1}\left(s\sum^{k}_{i=1}i^{s-1} \right )^{-p}.
\end{align}
   As it is easy to show that the function
\begin{align*}
   x \mapsto \frac {\left (1+\frac {1}{n} \right )^{x}-1}{x}
\end{align*}
   is an increasing function for fixed $n$, it follows that we
   only need to establish inequality \eqref{4.12} with $r$
   replaced by $1/p$. After simplification, it is equivalent to
   the following inequality:
\begin{align}
\label{4.14}
   \frac {\sum^{n+1}_{k=1}\left(s\sum^{k}_{i=1}i^{s-1} \right)^{-p}}{\sum^{n}_{k=1}\left(s\sum^{k}_{i=1}i^{s-1} \right)^{-p}}
   \geq \frac {(n+1)^{1-sp}}{n^{1-sp}}.
\end{align}
   In order to establish the above inequality, we first show that
   for any $n \geq 1$, we have
\begin{align*}
   \frac {\sum^{n+1}_{k=1}\left(s\sum^{k}_{i=1}i^{s-1} \right)^{-p}}{\sum^{n}_{k=1}\left(s\sum^{k}_{i=1}i^{s-1} \right)^{-p}}
   \geq \frac {\sum^{n+1}_{k=1}i^{-sp}}{\sum^{n}_{k=1}i^{-sp}}.
\end{align*}
    The case $n=1$ of the above inequality can be easily
    established by observing that $s \geq 1$. We now apply Lemma
    \ref{lem3} to conclude that it remains to show for any $n \geq
    1$,
\begin{align*}
   \frac {\left(\sum^{n+1}_{i=1}i^{s-1} \right)^{-p}}{\left(\sum^{n}_{i=1}i^{s-1} \right)^{-p}}
   \geq \frac {(n+1)^{-sp}}{n^{-sp}}.
\end{align*}
    The above inequality is equivalent to
\begin{align}
\label{4.15}
   \frac
{1}{n+1}\sum^{n+1}_{i=1}\left ( \frac {i}{n+1} \right )^{s-1} \leq
\frac {1}{n}\sum^{n}_{i=1}\left ( \frac {i}{n} \right )^{s-1}.
\end{align}
    To establish the above inequality, we define for any function $f(x)$ defined on the interval $(0,1]$ and any integer $n \geq 1$,
\begin{align*}
    R_{n}(f)=\frac {1}{n}\sum^{n}_{i=1}f(\frac {i}{n}).
\end{align*}
    Then a result \cite[Theorem 3A]{B&J} of Bennett and Jameson asserts that $R_n(f)$ decreases(resp. increases) with
    $n$ if $f(x)$ is an increasing (resp. decreasing) function which is either convex or
    concave. This result applied to the function $f(x)=x^{s-1}$
    leads immediately to inequality \eqref{4.15}.

    We now conclude that in order to show inequality \eqref{4.14},
    it remains to show that
\begin{align*}
   \frac {\sum^{n+1}_{k=1}i^{-sp}}{\sum^{n}_{k=1}i^{-sp}} \geq \frac {(n+1)^{1-sp}}{n^{1-sp}}.
\end{align*}
    The above inequality is equivalent to
\begin{align*}
   \frac
{1}{n+1}\sum^{n+1}_{i=1}\left ( \frac {i}{n+1} \right )^{-sp} \geq
\frac {1}{n}\sum^{n}_{i=1}\left ( \frac {i}{n} \right )^{-sp},
\end{align*}
   which also follows from the above mentioned result of Bennett and
   Jameson applied to $f(x)=x^{-sp}$.
\end{proof}
\section{Proof of Theorem \ref{thm2}}
\label{sec 3} \setcounter{equation}{0}

    Motivated by the proof of \cite[Lemma 8]{BGE1}, we show that Theorem \ref{thm2} is a
    consequence of the following
\begin{theorem}[{\cite[Theorem 2]{B}, \cite[Theorem 4]{BGE}}]
\label{thm3}
  Let $0 < q \leq p < \infty$ and $p \geq 1$. Let $(a_{n,k})_{n,k \in {\mathbb N}}$ be a non-negative
matrix, $(b_k)$ be a non-negative sequence and let $C > 0$. Then
\begin{equation}
\label{2}
  \left ( \sum^{\infty}_{n=1} \left ( \sum^{\infty}_{k=1}a_{n,k}x_k\right )^q \right
   )^{1/q} \geq C \left ( \sum^{\infty}_{n=1}b_nx^p_n  \right )^{1/p}
\end{equation}
   holds for all non-negative non-increasing sequences $(x_n)$ if and only
   if for all $m \in {\mathbb N}$,
\begin{equation}
\label{3}
    \left ( \sum^{\infty}_{n=1} \left ( \sum^{m}_{k=1}a_{n,k} \right )^q \right
   )^{1/q} \geq C \left ( \sum^{m}_{n=1}b_n  \right )^{1/p}.
\end{equation}
   The theorem continues to hold when $0 < p \leq q < \infty$ and $p \leq 1$ provided that
   inequalities \eqref{2} and \eqref{3} are reversed.
\end{theorem}

We may assume $p \geq 1, 0<q \leq p$ here as the proof for the
other case is similar.
    We denote $y_n=\sum^{\infty}_{k=n}c_kx_k,
    n \geq 1$ so that we have $y_1 \geq y_2 \geq \ldots \geq 0$, and that
\begin{equation*}
x_n=\frac {y_n-y_{n+1}}{c_n}.
\end{equation*}
    We can then recast inequality \eqref{1} as
\begin{align*}
\label{4}
   \left ( \sum^{\infty}_{n=1} b^p_ny^p_n \right )^{1/p} & \leq
   K(p,q)\left ( \sum^{\infty}_{n=1} \left ( \sum^{\infty}_{k=1}a_{n,k}\left ( \frac {y_k-y_{k+1}}{c_k} \right ) \right )^q \right
   )^{1/q} \\
   & = K(p,q)\left ( \sum^{\infty}_{n=1} \left ( \sum^{\infty}_{k=1}
   \left (\frac {a_{n,k}}{c_k}-\frac {a_{n,k-1}}{c_{k-1}} \right ) y_k \right )^q \right
   )^{1/q},
\end{align*}
   where we set $a_{n,0}/c_0=0$. Note that by our assumption, $a_{n,k}/c_k$ increases with $k$ for any
fixed $n \geq 1$, so that
\begin{equation*}
   \frac {a_{n,k}}{c_k}-\frac {a_{n,k-1}}{c_{k-1}} \geq 0.
\end{equation*}
   Now the assertion of Theorem \ref{thm2} readily follows from
   Theorem \ref{thm3}.

\section{Some Applications of Theorem \ref{thm2}}
\label{sec 4} \setcounter{equation}{0}
   In this section we look at some applications of Theorem
   \ref{thm2}. All of our results in this section are motivated by (the reversed) inequality \eqref{1.3} for $0<p<1$.
   Thus we assume $0<p<1$ throughout this section and let ${\bf a}=(a_n)$ be any non-negative sequence. We first apply Theorem
   \ref{thm2} with
\begin{align*}
 a_{n,k} = \left\{\begin{array}{ll}
\frac{n^{-r}}{k^{1-r}} & {\mbox{ if }}  k \geq  n,   \\
 0 & {\mbox{ if }}  1 \leq k <n
\end{array}\right., \ b_n=n^{-s}, \ c_k=1/k^{1-s}, \ r>s
\end{align*}
   to see that
\begin{equation*}
   \sum^{\infty}_{n=1}\left (n^{-r}\sum^{\infty}_{k=n}\frac
   {a_k}{k^{1-r}} \right )^p \leq \sup_m\left (\frac {\sum^{m}_{n=1}(n/m)^{-rp}}{\sum^{m}_{n=1}(n/m)^{-sp}} \right )
   \sum^{\infty}_{n=1}\left (n^{-s}\sum^{\infty}_{k=n}\frac
   {a_k}{k^{1-s}} \right )^p.
\end{equation*}

    It follows from Lemma \ref{lem1} that we have the following
\begin{theorem}
\label{thm4} For $0<p<1$, $s<r<1/p$, $a_n \geq 0$, we have
\begin{equation}
\label{3.4}
   \sum^{\infty}_{n=1}\left (n^{-s}\sum^{\infty}_{k=n}\frac
   {a_k}{k^{1-s}} \right )^p \leq \sum^{\infty}_{n=1}\left (n^{-r}\sum^{\infty}_{k=n}\frac
   {a_k}{k^{1-r}} \right )^p < \left (\frac {1-sp}{1-rp} \right
   )
   \sum^{\infty}_{n=1}\left (n^{-s}\sum^{\infty}_{k=n}\frac
   {a_k}{k^{1-s}} \right )^p.
\end{equation}
   The constants are best possible.
\end{theorem}

    Note that the first inequality in \eqref{3.4} follows as we
    have plainly for $k \geq n, r>s$, $n^{-s}/k^{1-s} \leq
    n^{-r}/k^{1-r}$. Upon taking $a_1=1, a_k=0, k \geq 2$, one
    sees that the first inequality in \eqref{3.4} is also best
    possible.

    Next, we apply Theorem \ref{thm2} with
\begin{align*}
 a_{n,k} = \left\{\begin{array}{ll}
\frac{n^{-r}}{(k+1)^{1-r}} & {\mbox{ if }}  k \geq  n,   \\
 0 & {\mbox{ if }}  1 \leq k <n
\end{array}\right., \ b_n=n^{-s}, \ c_k=1/(k+1)^{1-s}, \ r>s
\end{align*}
   to see that
\begin{equation*}
   \sum^{\infty}_{n=1}\left (n^{-r}\sum^{\infty}_{k=n}\frac
   {a_k}{(k+1)^{1-r}} \right )^p \leq \sup_m\left (\frac {\sum^{m}_{n=1}(n/(m+1))^{-rp}}{\sum^{m}_{n=1}(n/(m+1))^{-sp}} \right )
   \sum^{\infty}_{n=1}\left (n^{-s}\sum^{\infty}_{k=n}\frac
   {a_k}{(k+1)^{1-s}} \right )^p.
\end{equation*}

   It follows from Lemma \ref{lem2} that we have the following
\begin{theorem}
\label{thm5} For $0<p<1$, $s<r<1/p$, $a_n \geq 0$, we have
\begin{equation*}
   \sum^{\infty}_{n=1}\left (n^{-r}\sum^{\infty}_{k=n}\frac
   {a_k}{(k+1)^{1-r}} \right )^p \leq  C_{p,r,s}
   \sum^{\infty}_{n=1}\left (n^{-s}\sum^{\infty}_{k=n}\frac
   {a_k}{(k+1)^{1-s}} \right )^p,
\end{equation*}
    where
\begin{align}
\label{4.9}
   C_{p, r, s}=\left\{\begin{array}{ll}
2^{(r-s)p} & {\mbox{ if }}  \ s<r \leq -\frac {1}{p},   \\
 \frac {1-sp}{1-rp} & {\mbox{ if }}  \ 0<s<r < \frac
 {1}{p}.
\end{array}\right.
\end{align}
   The constant $C_{p,r,s}$ is best
possible.
\end{theorem}

\begin{cor}
\label{cor0}
   Let $a_n \geq 0,  0<p<1$. For $0< \beta<\alpha < \frac
 {1}{p}$, we have
\begin{equation}
\label{4.10}
   \sum^{\infty}_{n=1}\Big ( \frac {1}{n^{\alpha}}\sum^{\infty}_{k=n}\Big((k+1)^{\alpha}-k^{\alpha}\Big )a_k \Big
   )^p \leq \frac {\alpha^p}{\beta^p}C_{p,\alpha, \beta}\sum^{\infty}_{n=1}\Big ( \frac {1}{n^{\beta}}\sum^{\infty}_{k=n}\Big((k+1)^{\beta}-k^{\beta}\Big )a_k \Big
   )^p,
\end{equation}
    where $C_{p,\alpha, \beta}$ is defined as in \eqref{4.9} and the constant is best possible.
\end{cor}
\begin{proof}
   We apply Theorem \ref{thm2} with
\begin{align*}
 a_{n,k} = \left\{\begin{array}{ll}
 n^{-\alpha}\left ( (k+1)^{\alpha}-k^{\alpha} \right ) & {\mbox{ if }}  k \geq  n,   \\
 0 & {\mbox{ if }}  1 \leq k <n
\end{array}\right., \ b_n=n^{-\beta}, \ c_k=(k+1)^{\beta}-k^{\beta} , \
\alpha>\beta.
\end{align*}
   Note that the fact $a_{n,k}/c_k$ increases with $k$ is an easy consequence of the
   Mean Value Theorem. Thus we obtain
\begin{align*}
   &\sum^{\infty}_{n=1}\Big ( \frac {1}{n^{\alpha}}\sum^{\infty}_{k=n}\Big((k+1)^{\alpha}-k^{\alpha}\Big )a_k \Big
   )^p  \\
   \leq & \sup_m\left (\frac {\sum^{m}_{n=1}n^{-\alpha p}\left ((m+1)^{\alpha}-m^{\alpha} \right )^{p}}
   {\sum^{m}_{n=1}n^{-\beta p}\left ((m+1)^{\beta}-m^{\beta} \right )^{p}} \right )
   \sum^{\infty}_{n=1}\Big ( \frac {1}{n^{\beta}}\sum^{\infty}_{k=n}\Big((k+1)^{\beta}-k^{\beta}\Big )a_k \Big
   )^p.
\end{align*}
   Note that by the Mean Value Theorem, we have
\begin{align*}
   \frac
   {(m+1)^{\alpha}-m^{\alpha}}{(m+1)^{\beta}-m^{\beta}}=\frac
   {\alpha}{\beta}\xi^{\alpha-\beta}\leq \frac
   {\alpha}{\beta}(m+1)^{\alpha-\beta},
\end{align*}
   where $m<\xi<m+1$. It follows that
\begin{align*}
   \sup_m\left (\frac {\sum^{m}_{n=1}n^{-\alpha p}\left ((m+1)^{\alpha}-m^{\alpha} \right )^{p}}
   {\sum^{m}_{n=1}n^{-\beta p}\left ((m+1)^{\beta}-m^{\beta} \right )^{p}} \right )
   \leq \frac {\alpha^p}{\beta^p} \sup_m\left (\frac {\sum^{m}_{n=1}(n/(m+1))^{-\alpha p}}{\sum^{m}_{n=1}(n/(m+1))^{-\beta p}} \right ).
\end{align*}
    Inequality \eqref{4.10} then follows from Theorem \ref{thm5}.
    We further note that we have
\begin{align*}
   \lim_{m \rightarrow +\infty}\frac {\sum^{m}_{n=1}n^{-\alpha p}\left ((m+1)^{\alpha}-m^{\alpha} \right )^{p}}
   {\sum^{m}_{n=1}n^{-\beta p}\left ((m+1)^{\beta}-m^{\beta} \right )^{p}}
   =\frac {\alpha^p}{\beta^p} C_{p, \alpha, \beta}.
\end{align*}
   This shows that the constant in
   \eqref{4.10} is best possible and this completes the proof.
\end{proof}

    Upon letting $\beta \rightarrow 0^+$, we immediately obtain
    the following
\begin{cor}
\label{cor1}
   Let $a_n \geq 0,  0<p<1$. For $0< \alpha < \frac
 {1}{p}$, we have
\begin{equation*}
   \sum^{\infty}_{n=1}\Big ( \frac {1}{n^{\alpha}}\sum^{\infty}_{k=n}\Big((k+1)^{\alpha}-k^{\alpha}\Big )a_k \Big
   )^p \leq \left (\frac {\alpha^p}{1- \alpha p} \right )\sum^{\infty}_{n=1}\Big (\sum^{\infty}_{k=n}\ln \left ( \frac {k+1}{k}\right )a_k \Big
   )^p.
\end{equation*}
    The constant is best possible.
\end{cor}
    Note that as $\ln (1+1/k) \leq 1/k$, we have the following
\begin{cor}
\label{cor2}
   Let $a_n \geq 0,  0<p<1$. For $0< \alpha < \frac
 {1}{p}$, we have
\begin{equation*}
   \sum^{\infty}_{n=1}\Big ( \frac {1}{n^{\alpha}}\sum^{\infty}_{k=n}\Big((k+1)^{\alpha}-k^{\alpha}\Big )a_k \Big
   )^p \leq \left (\frac {\alpha^p}{1- \alpha p} \right )\sum^{\infty}_{n=1}\Big (\sum^{\infty}_{k=n}\frac {a_k}{k} \Big
   )^p.
\end{equation*}
    The constant is best possible.
\end{cor}

   We now consider an analogue of inequality \eqref{4.10} by
   taking
\begin{align*}
 a_{n,k} = \left\{\begin{array}{ll}
 k^{-\beta}\left ( n^{\beta}-(n-1)^{\beta} \right ) & {\mbox{ if }}  k \geq  n,   \\
 0 & {\mbox{ if }}  1 \leq k <n
\end{array}\right., \ b_n=n^{\alpha}-(n-1)^{\alpha} , \ c_k=k^{-\alpha}, \
\alpha>\beta.
\end{align*}
   Then it follows from Theorem \ref{thm2} that
\begin{align*}
  & \sum^{\infty}_{n=1}\left (\left ( n^{\beta}-(n-1)^{\beta} \right ) \sum^{\infty}_{k=n}k^{-\beta}a_k \right )^p \\
  \leq & \sup_m\left (\frac
{\sum^{m}_{n=1}(n^{\beta}-(n-1)^{\beta})^{p}m^{-\beta
p}}{\sum^{m}_{n=1}(n^{\alpha}-(n-1)^{\alpha})^{p}m^{-\alpha p}}
\right )
   \sum^{\infty}_{n=1}\left (\Big ( n^{\alpha}-(n-1)^{\alpha} \Big ) \sum^{\infty}_{k=n}k^{-\alpha}a_k \right )^p.
\end{align*}

   We recall that for two positive real finite sequences
   ${\bf x}=(x_1, x_2, \ldots, x_n)$ and ${\bf y}=(y_1, y_2,
   \ldots, y_n)$, ${\bf x}$ is said to be
   majorized by ${\bf y}$ if for all convex functions $f$, we have
\begin{equation*}
   \sum_{j=1}^{n}f(x_j) \leq \sum_{j=1}^{n}f(y_j).
\end{equation*}

    We write ${\bf x} \leq_{maj} {\bf y}$ if this occurs and the
    majorization principle states that if $(x_j)$ and $(y_j)$ are
    decreasing, then ${\bf x} \leq_{maj} {\bf y}$ is equivalent to
\begin{eqnarray*}
\label{10}
   x_1+x_2+\ldots+x_j & \leq & y_1+y_2+\ldots+y_j \ (1 \leq j \leq
   n-1),
     \\
    x_1+x_2+\ldots+x_n & = & y_1+y_2+\ldots+y_n .
\end{eqnarray*}
    We refer the reader to \cite[Sect. 1.30]{B&B} for a simple proof of
    this.

    Now suppose $0<\beta < \alpha \leq 1$, we apply the majorization principle to the convex function $-x^p$ and the two sequences
\begin{align*}
{\bf x}=\left(\frac {k^{\alpha}-(k-1)^{\alpha}}{n^{\alpha}}
\right)_{1 \leq k \leq n}, \ {\bf y}=\left(\frac
{k^{\beta}-(k-1)^{\beta}}{n^{\beta}} \right)_{1 \leq k \leq n}.
\end{align*}
   It's easy to see that both ${\bf x}$ and ${\bf y}$ are
   decreasing and ${\bf x} \leq_{maj} {\bf y}$. It follows that
\begin{align*}
   \frac
{\sum^{m}_{n=1}(n^{\beta}-(n-1)^{\beta})^{p}m^{-\beta
p}}{\sum^{m}_{n=1}(n^{\alpha}-(n-1)^{\alpha})^{p}m^{-\alpha p}}
\leq 1.
\end{align*}
   As the above inequality becomes an identity when $m=1$, we
   obtain the following
\begin{theorem}
\label{thm7}
   Let $a_n \geq 0,  0<p<1$. For $0< \beta < \alpha \leq 1$, we have
\begin{equation*}
\sum^{\infty}_{n=1}\left (\left ( n^{\beta}-(n-1)^{\beta} \right )
\sum^{\infty}_{k=n}k^{-\beta}a_k \right )^p
  \leq
   \sum^{\infty}_{n=1}\left (\Big ( n^{\alpha}-(n-1)^{\alpha} \Big ) \sum^{\infty}_{k=n}k^{-\alpha}a_k \right )^p.
\end{equation*}
    The constant is best possible.
\end{theorem}

   Now, we apply Theorem \ref{thm2} with
\begin{align*}
 a_{n,k} = \left\{\begin{array}{ll}
\frac{k^{r-1}}{\sum^{n}_{i=1}i^{r-1}} & {\mbox{ if }}  k \geq  n,   \\
 0 & {\mbox{ if }}  1 \leq k <n
\end{array}\right., \ b_n=\left (\sum^{n}_{i=1}i^{s-1} \right )^{-1}, \ c_k=k^{s-1}, \ r>s
\end{align*}
   to see that
\begin{align*}
  & \sum^{\infty}_{n=1}\left (\frac{1}{\sum^{n}_{i=1}i^{r-1}} \sum^{\infty}_{k=n}k^{r-1}a_k \right )^p  \\
  \leq & \sup_m\left (\frac
{\sum^{m}_{n=1}(\sum^{n}_{i=1}i^{r-1})^{-p}m^{(r-1)p}}{\sum^{m}_{n=1}(\sum^{n}_{i=1}i^{s-1})^{-p}m^{(s-1)p}}
\right )
   \sum^{\infty}_{n=1}\left (\frac{1}{\sum^{n}_{i=1}i^{s-1}} \sum^{\infty}_{k=n}k^{s-1}a_k \right )^p.
\end{align*}

  It follows from Lemma \ref{lem4} that we have the following
\begin{theorem}
\label{thm6} For $0<p<1$, $1 \leq s<r<1/p$, $a_n \geq 0$, we have
\begin{equation*}
   \sum^{\infty}_{n=1}\left (\frac{1}{\sum^{n}_{i=1}i^{r-1}} \sum^{\infty}_{k=n}k^{r-1}a_k \right )^p
   < \frac {r^p(1-sp)}{s^p(1-rp)}
    \sum^{\infty}_{n=1}\left (\frac{1}{\sum^{n}_{i=1}i^{s-1}} \sum^{\infty}_{k=n}k^{s-1}a_k \right
    )^p.
\end{equation*}
   The constant is best possible.
\end{theorem}

   We end this paper by considering an analogue to the above result. We apply Theorem \ref{thm2} with
\begin{align*}
 a_{n,k} = \left\{\begin{array}{ll}
\frac{n^{s-1}}{\sum^{k}_{i=1}i^{s-1}} & {\mbox{ if }}  k \geq  n,   \\
 0 & {\mbox{ if }}  1 \leq k <n
\end{array}\right., \ b_n=n^{r-1}, \ c_k=\left ( \sum^{k}_{i=1}i^{r-1} \right )^{-1}, \
r>s
\end{align*}
   to see that (note that the fact $a_{n,k}/c_k$ increases with
   $k$ follows from a simple application of Lemma \ref{lem3})
\begin{equation*}
   \sum^{\infty}_{n=1}\left (n^{s-1} \sum^{\infty}_{k=n}\frac{a_k}{\sum^{n}_{i=1}i^{s-1}} \right )^p \leq \sup_m\left (\frac
{\sum^{m}_{n=1}n^{(s-1)p}(\sum^{m}_{i=1}i^{s-1})^{-p}}{\sum^{m}_{n=1}n^{(r-1)p}(\sum^{m}_{i=1}i^{r-1})^{-p}}
\right )
   \sum^{\infty}_{n=1}\left (n^{r-1} \sum^{\infty}_{k=n}\frac{a_k}{\sum^{n}_{i=1}i^{r-1}} \right )^p.
\end{equation*}

    Suppose now $s<r \leq 1$, we apply the majorization principle again to the convex function $-x^p$ and the two sequences
\begin{align*}
{\bf x}=\left(\frac {k^{r-1}}{\sum^{n}_{i=1}i^{r-1}} \right)_{1
\leq k \leq n}, \ {\bf y}=\left(\frac
{k^{s-1}}{\sum^{n}_{i=1}i^{s-1}}  \right)_{1 \leq k \leq n}.
\end{align*}
   It's easy to see that both ${\bf x}$ and ${\bf y}$ are
   decreasing and ${\bf x} \leq_{maj} {\bf y}$ (for example, by an application of Lemma \ref{lem3}). It follows that
\begin{align*}
   \frac
{\sum^{m}_{n=1}n^{(s-1)p}(\sum^{m}_{i=1}i^{s-1})^{-p}}{\sum^{m}_{n=1}n^{(r-1)p}(\sum^{m}_{i=1}i^{r-1})^{-p}}
\leq 1.
\end{align*}
   As the above inequality becomes an identity when $m=1$, we
   obtain the following
\begin{theorem}
\label{thm8}
   Let $a_n \geq 0,  0<p<1$. For $s < r \leq 1$, we have
\begin{equation*}
 \sum^{\infty}_{n=1}\left (n^{s-1} \sum^{\infty}_{k=n}\frac{a_k}{\sum^{n}_{i=1}i^{s-1}} \right )^p \leq
   \sum^{\infty}_{n=1}\left (n^{r-1} \sum^{\infty}_{k=n}\frac{a_k}{\sum^{n}_{i=1}i^{r-1}} \right )^p.
\end{equation*}
    The constant is best possible.
\end{theorem}


\end{document}